\documentclass{amsart}

\usepackage[utf8]{inputenc}
\usepackage[T1]{fontenc}
\usepackage{verbatim}
\usepackage{graphicx}
\usepackage{graphicx,caption2,psfrag,float,color}
\usepackage{amssymb}
\usepackage{amscd}
\usepackage{amsmath}

\usepackage[T1]{fontenc}

\newtheorem{theorem}{Theorem}[section]

\newtheorem{lemma}[theorem]{Lemma}

\numberwithin{equation}{subsection}
\newtheorem{definition}[theorem]{Definition}

\title{On the maximum of cotangent sums related to the Riemann Hypothesis in rational numbers in short intervals}
\author{Helmut Maier and Michael Th. Rassias}
\date{\today}
\address{Department of Mathematics, University of Ulm, Helmholtzstrasse 18, 89081 Ulm, Germany.}
\email{helmut.maier@uni-ulm.de}
\address{Institute of Mathematics, University of Zurich,
 CH-8057, Zurich, Switzerland
  \& Moscow Institute of Physics and Technology
   141700 Dolgoprudny, Institutskiy per, d. 9, Russia
 \& Institute for Advanced Study, Program in Interdisciplinary Studies, 1
Einstein Dr, Princeton, NJ 08540, USA.}
\email{michail.rassias@math.uzh.ch, michailrassias@math.princeton.edu}
\thanks{}

\begin{document}

 \maketitle
 
\begin{abstract} 
Cotangent sums play a significant role in the Nyman-Beurling criterion for the Riemann Hypothesis. Here we investigate the maximum of the values of these cotangent sums over various sets of rational numbers in short intervals.

\end{abstract}
\section{Introduction}

The subject of this paper, the cotangent sums
\[
c_0\left(\frac{r}{b}\right):=-\sum_{m=1}^{b-1}\frac{m}{b}\cot\left(\frac{\pi m r}{b}\right)\:, \tag{1.1}
\]
has been studied by the authors in various papers (cf. \cite{mr}, \cite{mr2}, \cite{mr3}, \cite{mr4}, \cite{mr4.2}, \cite{mr5}, \cite{mrOnthesize}, \cite{mrest}, \cite{mrest2}, \cite{mr6}, \cite{mr7}, \cite{raigo_ras1}) and by the second author in his thesis \cite{rasthesis}. In \cite{rasthesis} the author considers moments of $c_0(r/b)$ as the variable $r$ ranges over the set
$$\{r\::\: (r,b)=1,\ A_0b\leq r\leq A_1b\}\:,$$
where $A_0, A_1$ are fixed with $1/2<A_0<A_1<1$ and $b$ tends to infinity. He could show that
$$\frac{1}{\phi(b)}\sum_{\substack{(r,b)=1\\ A_0b\leq r\leq A_1b}}c_0\left(\frac{r}{b}\right)^{2k}=H_kb^{2k}(1+o(1)),\ (b\rightarrow +\infty),$$
where 
$$H_k:=\int_0^1\left(\frac{g(x)}{\pi}\right)^{2k}dx\:,$$
$$g(x):=\sum_{l\geq 1}\frac{1-2\{lx\}}{l}\:.$$
The range $1/2<A_0<A_1<1$ was later extended to $0<A_0<A_1<1$ by S. Bettin in \cite{bettin}.\\
The cotangent sums $c_0(r/b)$ can be associated to the study of the Riemann Hypothesis  through its relation with the so-called Vasyunin sum $V$, which is defined as follows:
$$V\left(\frac{r}{b}\right):=\sum_{m=1}^{b-1}\left\{\frac{mr}{b}\right\}\cot\left(\frac{\pi mr}{b}\right)\:,$$
where $\{u\}:=u-\lfloor u \rfloor$, $u\in\mathbb{R}.$\\
It can be shown that 
$$V\left(\frac{r}{b}\right)=-c_0\left(\frac{\bar{r}}{b}\right),$$
where $\bar{r}$ is such that $\bar{r} r\equiv1\:(\bmod\; b)$.\\
The Vasyunin sum is itself associated to the study of the Riemann Hypothesis through the following identity (see \cite{BETT2}, \cite{BEC}):
\begin{align*}
&\frac{1}{2\pi(rb)^{1/2}}\int_{-\infty}^{+\infty}\left|\zeta\left(\frac{1}{2}+it\right)\right|^2\left(\frac{r}{b}\right)^{it}\frac{dt}{\frac{1}{4}+t^2}\\
&=\frac{\log 2\pi -\gamma}{2}\left(\frac{1}{r}+\frac{1}{b} \right)
+\frac{b-r}{2rb}\log\frac{r}{b}-\frac{\pi}{2rb}\left(V\left(\frac{r}{b}\right)+V\left(\frac{b}{r}\right)\right). \tag{1.2}
\end{align*}
According to this approach initiated by Nyman and Beurling, the  Riemann Hypothesis is true if and only if 
$$\lim_{N\rightarrow+\infty}d_N=0,$$
where
$$d_N^{2}:=\inf_{D_N}\frac{1}{2\pi}\int_{-\infty}^{+\infty}\left|1-\zeta\left(\frac{1}{2}+it\right) D_N\left(\frac{1}{2}+it\right)\right|^2\frac{dt}{\frac{1}{4}+t^2}$$
and the infimum is taken over all Dirichlet polynomials
$$D_N(s)=\sum_{n=1}^{N}\frac{a_n}{n^s}.$$
In the paper \cite{mr5} the authors investigate the maximum of $|c_0(r/b)|$ for fixed large $b$  and $r$ lying in a short interval $[A_0b, (A_0+\Delta)b]$, $0<A_0<1$, $\Delta=b^{-C}$, $0<C<1/2$ fixed. We recall the following definitions and results from \cite{mr5}:
\begin{definition}\label{def11} 
Let $0<A_0<1$, $0<C<1/2$. For $b\in\mathbb{N}$ we set
$$\Delta:=\Delta(b,C)=b^{-C}.$$
We set
$$M(b, C,A_0):=\max_{A_0b\leq r\leq (A_0+\Delta)b}\left|c_0\left(\frac{r}{b}\right)\right|\:.$$
\end{definition}
\begin{theorem}\label{thm1} (Theorem 1.2 of \cite{mr5})$ $\\
With Definition \ref{def11} let $D$ satisfy $0<D<\frac{1}{2}-C$. Then we have for sufficiently large $b$:
$$M(b,C,A_0)\geq \frac{D}{\pi}b\log b\:.$$
\end{theorem}
In this paper we modify this result in two directions:\\
I) We restrict the numerator $r$ in \cite{mr5} to the sequence of prime numbers. We shall prove:
\begin{theorem}\label{thm1.3}
Let $q$ be prime, $C, D>0$, $C+D<1/32$. Let $\Delta$ be defined as in Definition \ref{def11},
$$M_p(q, C, A_0):=\max_{\substack{A_0q\leq p \leq (A_0+\Delta)q \\ p\  \text{prime}}}  \left|c_0\left(\frac{p}{q}\right)  \right|\:.$$
Then we have for sufficiently large $q$:
$$M_p(q, C, A_0)\geq \frac{D}{\pi} q\log q\:.$$
\end{theorem}

\noindent II) We consider fractions $r/b$ simultaneously varying the numerator $r$ and the denominator $b$.

\begin{definition}\label{defn1.4}
For $\alpha\in(0,1), \Delta>0, B\in\mathbb{N}$ we define
$$\mathcal{R}(\alpha, \Delta, B):=\left\{ \frac{r}{b}\::\: \left|\alpha-\frac{r}{b}\right|<\Delta, (r,b)=1, B\leq b\leq 2B\right\}\:,$$
$$M_s(\alpha, \Delta, B):=\max_{\frac{r}{b}\in\mathcal{R}(\alpha,\Delta,B)} \left|c_0\left(\frac{r}{b}\right)\right|\:.$$
\end{definition}
We shall prove:
\begin{theorem}\label{thm1.5}
Let $\alpha\in(0,1)$, $0<C, D<1$, $C+D<3/4$. Then we have for $B$ sufficiently large:
$$M_s(\alpha, B^{-C}, B)\geq \frac{D}{\pi}\: B\log B\:.$$
\end{theorem}

Basic for the proof of the result of the paper \cite{mr5} as well as for the proofs of the results of the present paper is the relation of $c_0(r/b)$ to the Estermann zeta function $E\left(s, \frac{r}{b}, \alpha\right)$ and the closely related function $D_{sin}(s,x)$. We give the following definition and lemma.

\begin{definition}\label{defn1.6}
Let $Re\:s>Re\:\alpha+1$, $b\geq 1$, $(r,b)=1$ and
$$\sigma_{\alpha}(n):=\sum_{d|n}d^{\alpha}\:.$$
The Estermann zeta function is defined by 
$$E\left(s,\frac{r}{b},\alpha\right):=\sum_{n\geq 1}\frac{\sigma_{\alpha}(n)\exp\left(2\pi inr/b\right)}{n^s}\:. $$
For $x\in\mathbb{R}$, $Re\: s>1$, we set 
\[
D_{sin}(s,x):=\sum_{n\geq 1}\frac{d(n)\sin(2\pi n x)}{n^s} \:.\tag{1.3}
\]
\end{definition}

\begin{lemma}\label{lem1.7}
\[
c_0\left(\frac{r}{b}\right)=\frac{1}{2} D_{sin}\left(0, \frac{r}{b}\right)=2b\pi^{-2} D_{sin}\left(1,\frac{\bar{r}}{b}\right)\:.
\]
\end{lemma}
\begin{proof}
This is Lemma 2.6 of \cite{mr5}.
\end{proof}

From Lemma \ref{lem1.7} it becomes clear, that a crucial step in the proofs has to be the simultaneous localization of the fractions $r/b$ and $\bar{r}/b$. After confining $r/b$ and $\bar{r}/b$ to certain intervals and approximating the characteristic functions of these intervals by Fourier series, this leads to the problem of estimating certain exponential sums.\\
In \cite{mr5} Kloosterman sums with the fixed denominator $b$ are estimated by a result due to A. Weil. If the numerators $r$ are restricted to special subsets of the integers, like prime numbers, other exponential sums - in the present paper exponential sums in finite fields - must be considered. We apply estimates due to Fouvry and Michel \cite{fouvry_michel}.\\
If both numerators and denominators  are variable, sums of Kloosterman sums have to be considered. We shall apply results based on the Spectral Theory of Automorphic Forms due to Deshouillers and Iwaniec \cite{de_iw}.\\
In Section 4 we collect all definitions and results on these exponential sums, needed for the proofs of Theorems \ref{thm1.3} and \ref{thm1.5}.

\section{Preliminary Lemmas}
In the following lemma we give a relation between the value $D_{sin}(1,x)$ and the continued fraction expansion of $x$.

\begin{lemma}\label{lem2.1} Let $x=\left<a_0;a_1, a_2,\ldots\right>$ be the continued fraction expansion of $x\in\mathbb{R}$. Moreover, let $u_r/v_r$ be the $r$-th partial quotient of $x$. Then
\[
D_{sin}(1,x)=-\frac{\pi^2}{2}\sum_{l\geq 1}\frac{(-1)^l}{v_l}\left(\left(\frac{1}{\pi v_l}\right)+\psi\left(\frac{v_{l-1}}{v_l}\right)\right),\tag{2.1}
\]
whenever either of the two series (1.3), (2.1) is convergent.\\
If $x=\left<a_0;a_1, a_2,\ldots, a_r\right>$ is a rational number then the range of summation of the series on the right is to be interpreted to be $1\leq l\leq r$. Here $\psi$ is an 
analytic function satisfying
$$\psi(x)=-\frac{\log(2\pi x)-\gamma}{\pi x}+O(\log x),\ (x\rightarrow 0)\:.$$
\end{lemma}
\begin{proof}
This is Lemma 2.5 of \cite{mr5}.
\end{proof}

\begin{lemma}\label{lem2.2}
Let $\epsilon>0$, $b\geq b(\epsilon)$, $(r,b)=1$, $0< r< b$. Let 
$$\frac{r}{b}=\left<0;w_1,\ldots, w_s\right>$$
be the continued fraction expansion of $r/b$ with partial fractions $u_i/v_i$. Then there are at most 3 values of $l$ for which 
$$\frac{1}{v_l}\psi\left(\frac{v_{l-1}}{v_l}\right)\geq \log\log b$$
and at most one value of $l$, for which 
$$\frac{1}{v_l}\psi\left(\frac{v_{l-1}}{v_l}\right)\geq \epsilon\log b.$$
\end{lemma}
\begin{proof}
This is Lemma 2.14 of \cite{mr5}.
\end{proof}

\section{Fourier Analysis}

\begin{definition}\label{defn3.1}
For $\beta\in\mathbb{R}$, $v\geq 0$, $\Delta>0$, let the functions 
$\chi_1, \chi_2$ be defined by
\begin{equation}
\chi_1(u,v):=\left\{
\begin{array}{l l}
   1\:, & \quad \text{if}\ \ \beta+v<u\leq \beta+\Delta-v\\
    0\:, & \quad \text{otherwise}\:\\
  \end{array} \right.
 \nonumber
\end{equation}
and
$$\chi_2(u):=\Delta^{-1}\int_0^{\Delta}\chi_1(u,v)dv\:.$$
\end{definition} 

\begin{lemma}\label{lem28}
We have
$$\chi_2(u)=\sum_{n=-\infty}^{\infty}a(n)e(nu),$$
where $a(0)=\Delta/2$ and 
\begin{equation}
a(n)=\left\{
\begin{array}{l l}
   O(\Delta)\:, & \quad \text{if}\ \ |n|\leq \Delta^{-1}\\
    O(\Delta^{-1}n^{-2})\:, & \quad \text{if}\ \ |n|>\Delta^{-1}\:.\\
  \end{array} \right.
 \nonumber
\end{equation}
\end{lemma}
\begin{proof}
This is Lemma 2.10 of \cite{mr5}.
\end{proof}
\begin{definition}\label{defini29}
For $\gamma>0$, $v\geq 0$, let
\begin{equation}
\chi_3(u,v):=\left\{
\begin{array}{l l}
   1\:, & \quad \text{if}\ \  -\gamma+v<u<\gamma\\
    0\:, & \quad \text{otherwise}\:\\
  \end{array} \right.
 \nonumber
\end{equation}
and
$$\chi_4(u):=\gamma^{-1}\int_0^\gamma \chi_3(u,v)dv\:.$$
\end{definition}
\begin{lemma}\label{lem210}
We have 
$$\chi_4(u)=\sum_{n=-\infty}^{+\infty} c(n)\:e(nu),$$
where $c(0)=\gamma$ and 
\begin{equation}
c(n)=\left\{
\begin{array}{l l}
   O(\gamma)\:, & \quad \text{if}\ \ |n|\leq \gamma^{-1}\\
    O(\gamma^{-1}n^{-2})\:, & \quad \text{if}\ \ |n|>\gamma^{-1}\:.\\
  \end{array} \right.
 \nonumber
\end{equation}
\end{lemma}
\begin{proof}
This is Lemma 2.22 of \cite{mr5}.
\end{proof}

\section{Exponential Sums}

\begin{definition}\label{def24} Let $b\in\mathbb{N}$, $m,n\in\mathbb{Z}$.
The Kloosterman sum $K(m,n,b)$ is defined by
$$K(m,n,b):=\sum_{\substack{r=1\\(r,b)=1}}^{b-1}e\left(\frac{mr+n\bar{r}}{b}\right)\:.$$
For $m=0$ (resp. $n=0$) we obtain the Ramanujan sums $K(0,n,b)$ (resp. $K(m,0,b)$).
\end{definition}
\begin{lemma}\label{lem4.2} We have the bounds
\[
|K(m,n,b)|\leq d(b)(m,n,b)^{1/2}b^{1/2}\tag{4.1}
\]
and
\[
|K(0,n,b)|\leq (n,b).\tag{4.2}
\]
\end{lemma}
\begin{proof}
The result (4.1) is due to Weil (cf. \cite{weil2}). The result (4.2) is elementary. 
\end{proof}

The next result has not been used in previous papers of the authors. It is due to Deshouillers and Iwaniec and is related to the Spectral Theory of Automorphic Forms.

\begin{lemma}\label{lem4.3}
For positive real numbers $T, M, N, \epsilon$ and complex sequences \mbox{$\vec{a}=(a_m)_{m\in\mathbb{N}}$,} \mbox{$\vec{b}=(b_n)_{n\in\mathbb{N}}$} one has 
\[
\sum_{M<m\leq 2M} a_m \sum_{N<n\leq 2N} b_n \sum_{b\leq \left(\frac{mn}{MN}\right)^{\frac{1}{2}} T} \frac{1}{b}\: K(m, \pm n, b)\tag{4.2}
\]
$$\ll T^\epsilon\left\{ (MN)^{1/2}+(TMN)^{1/6} \right\} \|a_M\|_2 \|b_N\|_2\:. $$
The constant implied in $\ll$ depends on $\epsilon$ alone. Here $\|c_N\|_2$ is defined by 
$$\|c_N\|_2:=\left( \sum_{N<n\leq 2N} |c_N|^2 \right)^{1/2}\:.$$
\end{lemma}
\begin{proof}
This is formula (1.47) from the Corollary to Theorem 8 in \cite{de_iw}.
\end{proof}

\begin{definition}\label{defn4.4} Let $\mathbb{F}_q$ be the finite field with $q$ elements and let $\psi$ be a non-trivial additive character over $\mathbb{F}_q$, $f$ a rational function of the form
$$f(x)=\frac{P(x)}{Q(x)}\:,$$
$P$ and $Q$ relatively prime, monic non-constant polynomials,
$$S(f; q,x):=\sum_{p\leq x} \psi(f(p))\:,$$
($p$ denotes the $p$-fold sum of the element 1 in $\mathbb{F}_q$).
\end{definition}

\begin{lemma}\label{lem4.5}
With conditions from Definition \ref{defn4.4} we have:
$$S(f; q, x)\ll q^{3/16+\epsilon} x^{25/32}\:.$$
The implied constant depends only on $\epsilon$ and the degrees of $P$ and $Q$.
\end{lemma}
\begin{proof}
This is due to Fouvry and Michel \cite{fouvry_michel}.
\end{proof}

\section{Proof of Theorem \ref{thm1.3}}

\begin{definition}\label{defn5.1}
Let $\Delta=b^{-C}$ as in Definition \ref{def11} and let $\Omega>0$. Let $q$ be a prime number. We set
$$N(q, \Delta, \Omega):=\{p\ \text{prime}\::\: A_0q\leq p\leq  (A_0+\Delta)q,\ |\bar{p}|\leq\Omega  \}.$$
\end{definition}

\begin{definition}\label{defn5.2}
Let $S(f; q, x)$ be as in Definition \ref{defn4.4}. Let 
$E(m, n, q):=S(f;q,q)$ with $f(p):=mp+\frac{n}{p}$, $\psi(u):=e(u/q)$.
\end{definition}

\begin{lemma}\label{lem5.3}
We have 
$$E(m,n,q)\ll q^{31/32+\epsilon}$$ for all $\epsilon>0$.
\end{lemma}
\begin{proof}
This follows from Lemma \ref{lem4.5} with $x=q$.
\end{proof}

\begin{lemma}\label{lem5.4}
We have 
$$N(q,\Delta, \Omega)>0,$$
for $q$ sufficiently large.
\end{lemma}
\begin{proof}
By Definitions \ref{defn3.1}, \ref{defini29},  \ref{defn5.1} and Lemmas  \ref{lem28}, \ref{lem210}, \ref{lem5.3} we have
\begin{align*}
N(q, \Delta, \Omega)&\geq  \phi(q)a(0)c(0)+\sum_{m,n=-\infty}^\infty a(m)c(n)|E(m,n,q)|\\
&\geq \phi(q) q^{-(C+D+\epsilon)}+O(q^{31/32+\epsilon}),
\end{align*}
which proves the result.
\end{proof}

We may now conclude the result of Theorem \ref{thm1.3}. By Lemma \ref{lem5.4} there is at least one prime $p\in[A_0q, (A_0+\Delta)q]$, such that $\frac{\bar{p}}{q}\in (0, \Omega)$.\\
By Lemma \ref{lem2.1} we have:
$$c_0\left( \frac{p}{q} \right)=-q\sum_{l\geq 1} \frac{(-1)^l}{v_l}\left(\frac{1}{\pi v_l}+\psi\left(\frac{v_{l-1}}{v_l}\right)\right)\:.$$
Let $(u_i/v_i)_{i=1}^s$ be the sequence of partial fractions of $\frac{\bar{p}}{q}$. From $$\Omega \geq \frac{\bar{p}}{q} \geq \frac{1}{v_1+1}$$ we obtain $v_1+1\geq \Omega^{-1}$.\\
By Lemma \ref{lem2.2} we have
$$\sum_{l>1} \left(\frac{1}{\pi v_l}+\psi\left(\frac{v_{l-1}}{v_l}\right)\right)< 2\epsilon \log q\:,\ \ \text{for}\ q\geq q_0(\epsilon)\:.$$
Therefore,
$$\left|D_{sin}\left(0,\frac{p}{q}\right)\right|\geq \frac{1}{\pi}\log (\Omega^{-1}(1+o(1))\:,\ \ (q\rightarrow \infty)\:.$$
This proves Theorem \ref{thm1.3}.

\section{Proof of Theorem \ref{thm1.5}}

\begin{definition}\label{defn6.1}
Let $\alpha\in(0,1)$, $\Delta>0$, $\Omega>0$. We set
$$N(\alpha,\Delta, \Omega):=\#\{(b,r)\::\: \alpha\leq \frac{r}{b}\leq \alpha+\Delta, (r, b)=1, |\bar{r}|\leq \Omega b, B<b\leq 2B\}.$$
\end{definition}

\begin{lemma}\label{lem62}
Let $0<C, D<1$, $C+D<3/4$. Then we have:
$$N(\alpha, \Delta, \Omega)>0\:.$$
\end{lemma}
\begin{proof}
By Definition \ref{defn6.1}, Lemmas \ref{lem28}, \ref{lem210} we have with a positive constant $c^*>0$:
\[
N(\alpha, \Delta, \Omega)\geq c^* B^2 a(0)c(0)+\sum_{\substack{(m,n)=-\infty \\ (m,n)\neq (0,0)}}^\infty a(m) c(n) \sum_{B< b\leq 2B} K(m,n,b) \tag{6.1}
\]
It suffices to treat only the terms with $m>0$, $n>0$. We partition the sum into subsums:
$$\Sigma_{M,N}:=\sum_{M<m\leq 2M} a_m\sum_{N<n\leq 2N} c_n\sum_{B<b\leq 2B} K(m,n,b)\:.$$
We obtain $\Sigma_{M,N}$ by partial summation from 
$$\Sigma_{M,N,u}:=\sum_{M<m\leq 2M} a_m\sum_{N<n\leq 2N} c_n \sum_{1\leq b\leq u}\frac{1}{b} K(m,n,b)\:.$$
We choose integers $\mu=\mu(M)$ and $\nu=\nu(N)$ to be determined later and partition the interval $(M,2M]$ into $O(\mu)$ subintervals $I_k:=(m_k, m_{k+1}]$ of lengths $|I_k|$ with 
$$\frac{1}{2}\mu^{-1} M<|I_k|\leq 2\mu^{-1} M\:,\ \ 1\leq k\leq k_0(M)$$
and the interval $(N, 2N]$ into $O(v)$ subintervals $J_l:=(n_l, n_{l+1}]$ of lengths $|J_l|$ with 
$$\frac{1}{2}\nu^{-1}N<|J_l|\leq 2\nu^{-1} N,\ \ 1\leq l\leq l_0(N)\:.$$
We have 
$$\Sigma_{M,N,u}=\sum_{1\leq k\leq k^*(M)}\sum_{1\leq l\leq l^*(N)}\Sigma_{M,N,u}^{(k,l)}$$
where 
$$\Sigma_{M,N,u}^{(k,l)}=\sum_{m\in I_k}a_m\sum_{n\in J_l}c_n\sum_{1\leq b\leq u}\frac{1}{b} K(m,n,b)\:.$$
We now partition the sums $\Sigma_{M,N,u}^{(k,l)}$ into two subsums. For this purpose we define $T=T(k,l,u)$ by 
$$T\left(\frac{m_k n_l}{MN} \right)^{\frac{1}{2}}=u$$
and set 
$$\Sigma_{M,N,u}^{(k,l)}:=\Sigma_{M,N,u}^{(k,l,1)}-\Sigma_{M,N,u}^{(k,l,2)}\:,$$
where
$$\Sigma_{M,N,u}^{(k,l,1)}:=\sum_{m\in I_k}a_m\sum_{n\in J_l}c_n\sum_{1\leq b\leq T\left(\frac{mn}{MN} \right)^{\frac{1}{2}}} \frac{1}{b}\:K(m,n,b)$$
$$\Sigma_{M,N,u}^{(k,l,2)}:=\sum_{m\in I_k}a_m\sum_{n\in J_l}c_n\sum_{u\leq b\leq T\left(\frac{mn}{MN} \right)^{\frac{1}{2}}} \frac{1}{b}\: K(m,n,b)$$

We write $M=B^C2^\kappa$, $N=B^D 2^\lambda$, with $\kappa, \lambda\in\mathbb{Z}$.\\
We apply Lemma \ref{lem4.3} with the sequences $(a_{m,k})$, $(b_{n,l})$ in place of $\vec{a}$ and $\vec{b}$, which we denote by
\begin{equation}
a_{m,k}:=\left\{
\begin{array}{l l}
   a_m\:, & \quad \text{if}\ \ m\in I_k\\
    0\:, & \quad \text{otherwise}\:,\\
  \end{array} \right.
 \nonumber
\end{equation}
\begin{equation}
b_{n,l}:=\left\{
\begin{array}{l l}
   b_n\:, & \quad \text{if}\ \ n\in J_l\\
    0\:, & \quad \text{otherwise}\:.\\
  \end{array} \right.
 \nonumber
\end{equation}
We obtain by Lemmas \ref{lem4.3}, \ref{lem28}, \ref{lem210}:
\begin{equation}
\tag{6.2} \|a_{m,k}\|_2=\left\{
\begin{array}{l l}
   O\left(\mu^{-\frac{1}{2}}B^{-\frac{C}{2}}2^{\frac{\kappa}{2}}\right)\:, & \quad \text{if}\ \ \kappa\leq 0\\
    O\left(\mu^{-\frac{1}{2}}B^{-\frac{C}{2}}2^{-\kappa}\right)\:, & \quad \text{if}\ \ \kappa> 0\:,\\
  \end{array} \right.
 \nonumber
\end{equation}
\begin{equation}
\tag{6.3} \|b_{n,l}\|_2=\left\{
\begin{array}{l l}
   O\left(\nu^{-\frac{1}{2}}B^{-\frac{D}{2}}2^{\frac{\lambda}{2}}\right)\:, & \quad \text{if}\ \ \lambda\leq 0\\
    O\left(\nu^{-\frac{1}{2}}B^{-\frac{D}{2}}2^{-\lambda}\right)\:, & \quad \text{if}\ \ \lambda> 0\:.\\
  \end{array} \right.
 \nonumber
\end{equation}
We now estimate $\Sigma_{M,N,u}^{(k,l,2)}$. We let
\[
\mathcal{J}_{(k,l,u)}^{M,N}:=\{(m,n,b)\::\: m\in I_k, n\in J_l, u<b\leq T(m,n,u)  \}\:.  \tag{6.4}
\]
We have
\[
|\mathcal{J}_{(k,l,u)}^{M,N}| \ll 2^{\kappa+\lambda} B^{1+C+D}(\mu^{-1}+\nu^{-1})\mu^{-1}\nu^{-1}.  \tag{6.5}
\]
For $(m,n,b)\in \mathcal{J}_{(k,l,u)}^{M,N}$ we estimate the sums $K(m,n,b)$ individually by the use of Lemma \ref{lem4.2} and obtain by (6.2), (6.3), (6.4) and (6.5):
\[
\sum_{\substack{M,N\\k,l}}\ \sum_{(m,n,b)\in\mathcal{J}(M,N,k,l,u)}a_{m,k} c_{n,l} K(m,n,b)=O\left(B^{\frac{3}{2}+\epsilon} (\mu^{-1}+\nu^{-1})\right)\:.  \tag{6.6}
\]
From (6.2), (6.3), (6.6) we finally get:
$$N(\alpha, \Delta, \Omega)=c^*B^2a(0)c(0)+O\left(B\mu^{\frac{1}{2}}\nu^{\frac{1}{2}}\right)+O\left(B^{\frac{3}{2}}\left(\mu^{-1}+\nu^{-1}\right)\right)\:.$$
We have 
$$a(0)c(0)\geq B^{-C-D+\epsilon}\:.$$
We choose $\mu=\nu=B^{\frac{1}{4}}$ and obtain the proof of Lemma \ref{lem62} for the case $C>1/4$.\\
For the case $C\leq 1/4$ we only partition the interval $(N,2N]$ and sum over the contributions of the different values of $m$. Instead of $T=T(k,l,u)$, defined by
$$T\left( \frac{m_k u_l}{MN}\right)^{\frac{1}{2}}=u$$
we define $T=T(l,u)$ by 
$$T\left( \frac{ n_l}{N}\right)^{\frac{1}{2}}=u\:.$$
$\mathcal{J}_{(k,l,u)}^{M,N}$ in (6.4) is replaced by
$$\mathcal{J}_{(l,u)}^{N}:=\{ (n,b)\::\: n\in J_l, u<b\leq T(l,u) \}\:.$$
We again estimate the Kloosterman sums for $(n,b)\in \mathcal{J}_{(l,u)}^{N}$ individually by Lemma \ref{lem4.2} and for the other pairs $(n,b)$ by Lemma \ref{lem4.3}. This proves Lemma \ref{lem62} also for $C\leq 1/4$. The case $D\leq 1/4$ is analogous.\\
Thus the proof of Lemma \ref{lem62} is finished. 
\end{proof}

We may now conclude the proof of Theorem \ref{thm1.5}.\\
By Lemma \ref{lem62} there is at least one pair $(b,r)$, such that $(r,b)=1$,
$$\alpha \leq \frac{1}{b}< \alpha +\Delta,\ \ \left|\frac{\bar{r}}{b}\right|\leq \Omega,\ \  B< b\leq 2B. $$
Let $(u_i/v_i)_{i=1}^s$ be the sequence of partial fractions of $\bar{r}/b$. From 
$$\Omega\geq \frac{\bar{r}}{b}\geq \frac{1}{v_1+1}$$
we obtain $v_1+1\geq \Omega^{-1}$.\\
By Lemma \ref{lem2.1} we have:
$$c_0\left( \frac{r}{b} \right)=-b\sum_{l\geq 1} \frac{(-1)^l}{v_l}\left(\frac{1}{\pi v_l}+\psi\left(\frac{v_{l-1}}{v_l}\right)\right)\:.$$
By Lemma \ref{lem2.2} we have
$$\sum_{l \geq 1} \left(\frac{1}{\pi v_l}+\psi\left(\frac{v_{l-1}}{v_l}\right)\right)< 2\epsilon \log B\:,\ \ \text{for}\ B\geq B_0(\epsilon)\:.$$
Therefore,
$$\left|D_{sin}\left(0,\frac{r}{b}\right)\right|\geq \frac{1}{\pi}\log (\Omega^{-1}(1+o(1))\:,\ \ (B\rightarrow \infty)\:.$$
This proves Theorem \ref{thm1.5}.  \qed

%
%
%


%
%
%
%
%
%
%
%

\end{document}